\documentclass{amsart}
\usepackage{amsfonts}

\newtheorem{theorem}{Theorem}
\theoremstyle{plain}

\newtheorem{definition}{Definition}

\numberwithin{equation}{section}

\begin{document}
\title[Optimal linear approximation]{Optimal linear approximation and isometric extensions}
\author{Alexander Kushpel}
\address{\c{C}ankaya University, Department of Mathematics\\
\qquad Ankara, Turkey}
\email{kushpel@cankaya.edu.tr}

\begin{abstract}
Let $X$ be a Banach space with the unit ball $B(X)$ and $A\subset X$ be a
convex origin-symmetric compact in $X$. Let $\mathrm{j}:X\rightarrow 
\widetilde{X}$ be an isometric extension of $X$. It is well-known that
linear widths $\lambda _{n}\left( \mathrm{j}\left( A\right) \text{,}%
\widetilde{X}\right) $ may decrease in order when compared with $\lambda
_{n}\left( A\text{,}X\right) $ and absolute widths $\Lambda \left( A,%
\widehat{X}\right) =\inf_{\mathrm{j}}\left( \mathrm{j}\left( A\right) ,%
\widetilde{X}\right) $ are realized in the space $\widehat{X}$ which is the
Banach space of bounded functions $f:B\left( X^{\ast }\right) \rightarrow 
\mathbb{R}$ on the unit ball $B\left( X^{\ast }\right) $ of the conjugate
space $X^{\ast }$. We show that it is sufficient to use just $n$-dimensional
extensions of $X$ to attain absolute linear widths. This unexpected fact
significantly reduces the space $\ \widehat{X}$. This allows us to introduce
the notion of preabsolute widths. We give the respective optimal extensions
explicitly and establish order estimates for preabsolute widths of a wide
range of sets of smooth functions considered in \cite{C11}. In particular,
in the case of super-small and super-high smoothness considered in \cite{C11}
the orders of preabsolute linear widths coincide with the orders of absolute
linear widths. In the intermediate cases of finite and infinite smoothness
the respective orders are different.
\end{abstract}

\maketitle

\section{Introduction}

Optimal linear approximation and recovery play an important role in
Approximation Theory and they are connected via absolute linear widths and
duality with nonlinear approximation. $n$-Widths were introduced in 1936 by
Kolmogorov to compare the efficiency of numerical algorithms \cite{C2}. Let $%
\left( X,\left\Vert \cdot \right\Vert _{X}\right) $ be a Banach space with
the unit ball $B\left( X\right) $ and $A\subset X$ be a compact, convex and
origin symmetric set in $X$. The Kolmogorov $n$-width of $A$ in $X$ \ is
defined as%
\begin{equation*}
d_{n}\left( A,X\right) =\inf_{L_{n}\subset X}\sup_{x\in A}\inf_{y\in
L_{n}}\left\Vert x-y\right\Vert _{X}
\end{equation*}%
Let%
\begin{equation*}
d^{n}\left( A,X\right) =\inf_{L^{n}}\sup_{x\in A\cap L^{n}}\left\Vert
x\right\Vert _{X},
\end{equation*}%
be the Gelfand $n$-width \cite{C1}. Here $L^{n}$ runs over all subspaces of
codimension at most $n$. We shall concentrate here on linear widths
introduced in \cite{C3}. The linear $n$-width of $A$ in $X$ is defined by%
\begin{equation*}
\lambda _{n}\left( A,X\right) =\inf_{\mathrm{P}_{n}}\sup_{x\in A}\left\Vert
x-\mathrm{P}_{n}x\right\Vert _{X},
\end{equation*}%
where $\mathrm{P}_{n}:X\rightarrow X$ varies over all linear operators of
rank at most $n$. Let $X$ and $Z$ be Banach spaces, $u:X\rightarrow Z$, $%
u\in \mathcal{L}\left( X,Z\right) $ be a bounded linear operator and $%
u^{\ast }$ be its adjoint. It is well-known if $u$ is compact or $Z$ is
reflexive (see e.g. \cite{MAKOV}, \cite{KUSHP-0}) then 
\begin{equation}
d^{n}\left( u^{\ast }\right) =d_{n}\left( u\right)   \label{dl1}
\end{equation}%
and%
\begin{equation*}
\lambda _{n}\left( u\right) =\lambda _{n}\left( u^{\ast }\right) ,
\end{equation*}%
where%
\begin{equation*}
d_{n}\left( u\right) =d_{n}\left( uB\left( X\right) ,Z\right)
=\inf_{L_{n}\subset X}\sup_{x\in B\left( X\right) }\inf_{y\in
L_{n}}\left\Vert ux-y\right\Vert _{X},
\end{equation*}%
\begin{equation*}
\lambda _{n}\left( u\right) =\lambda _{n}\left( uB\left( X\right) ,Z\right)
=\inf_{\mathrm{P}_{n}}\sup_{x\in B\left( X\right) }\left\Vert ux-\mathrm{P}%
_{n}ux\right\Vert _{X}.
\end{equation*}

Let $\left( \widetilde{X},\mathrm{j}\right) $ be an extension of $X\subset 
\widetilde{X}$, where $\mathrm{j}:X\rightarrow \widetilde{X}$ is a linear
isometry. It was noticed by Kolmogorov and demonstrated on a concrete
example by Tikhomirov \cite{C3} that the linear $n$-width of $A$ in $X$ may
decrease in an isometric extension $\widetilde{X}$ of $X\subset $ $%
\widetilde{X}$, since $\widetilde{X}$ contains more subspaces to approximate 
$A$. Hence it is natural to consider 
\begin{equation*}
\Lambda _{n}\left( A,X\right) =\inf \lambda _{n}\left( \mathrm{j}\left(
A\right) ,\widetilde{X}\right) ,
\end{equation*}%
where $\inf $ is taken over all isometric extensions $\mathrm{j}%
:X\rightarrow \widetilde{X}$. The width $\Lambda _{n}\left( A,X\right) $ is
the absolute linear $n$-width introduced by Ismagilov \cite{C1}. It is known
that $\Lambda _{n}\left( A,X\right) =d^{n}\left( A,X\right) $. Moreover, the
absolute linear width is realized in the so-called universal isometric
extension $\widehat{X}$ which can be constructed as following. Let $B\left(
X^{\ast }\right) $ be the unit ball in the dual of $X$ and $\widehat{X}$ be
the Banach space of bounded functions $f:B\left( X^{\ast }\right)
\rightarrow \mathbb{R}$ with the usual norm 
\begin{equation*}
\left\Vert f\left( \phi \right) \right\Vert _{\widehat{X}}=\sup_{\phi \in
B\left( X^{\ast }\right) }\left\vert f\left( \phi \right) \right\vert .
\end{equation*}%
Clearly, $f\left( \phi \right) =\left\langle x,\phi \right\rangle \in 
\widehat{X}$ for any $x\in X$. By this way we get the linear isometric
extension $\mathrm{j}:X\rightarrow \widehat{X}$. Observe that Gelfand $n$%
-widths are closely connected to the linear cowidths. Let $\mathbb{R}^{n}$
be the coding set, i.e. the set which contains information on the elements
of $A$ and $\mathcal{L}\left( \mathrm{lin}\left( A\right) ,\mathbb{R}%
^{n}\right) $ be a family of coding operators, $\phi :A\rightarrow \mathbb{R}%
^{n}$. Let $D\subset X$,%
\begin{equation*}
\mathrm{diam}\left( D,X\right) =\sup \left\{ \left\Vert x-y\right\Vert
_{X}\left\vert x,y\in D\right. \right\}
\end{equation*}%
and 
\begin{equation*}
\phi ^{-1}\left( z\right) =\left\{ y\left\vert y\in X\text{, }\phi \left(
y\right) =z\right. \right\}
\end{equation*}%
be the diameter of $D$ in $X$ and preimage of $z\in X$ respectively. The
linear cowidth is defined as%
\begin{equation*}
\lambda ^{n}\left( A,X\right) =\inf_{\phi \in \mathcal{L}\left( \mathrm{lin}%
\left( A\right) ,\mathbb{R}^{n}\right) }\sup_{x\in A}\mathrm{diam}\left\{
\phi ^{-1}\left( \phi \left( x\right) \right) \right\} .
\end{equation*}%
Clearly, 
\begin{equation*}
\lambda ^{n}\left( A,X\right) =2d^{n}\left( A,X\right) .
\end{equation*}%
In Section \ref{Sec:results} we demonstrate an unexpected phenomenon.
Namely, we show that instead of a considerably big extension $\widehat{X}$
of $X$ it is sufficient to use $n$-dimensional extensions constructed in
Theorem \ref{THEOREM 1} to attain absolute $n$-widths $\Lambda _{n}\left(
A,X\right) $. Hence it is natural to introduce a new notion of preabsolute $%
n $-widths, $\Lambda _{n,m}\left( A,X\right) $ (see Definition \ref{DEF 1}),
where we allow $m$-dimensional isometric extensions of $X$, $0\leq m\leq n$.
\ In Section \ref{Sec:examples} we present two-side estimates for
preabsolute $n$-widths $\Lambda _{n,m}\left( A,X\right) $ on a wide range of
sets of smooth functions considered in \cite{C11}. More precisely, denote by 
$\mathcal{T}_{n}$, $n\in \mathbb{N}$ the sequence of subspaces of
trigonometric polynomials with the usual order, i.e. $\mathcal{T}_{n}=%
\mathrm{lin}\left\{ 1,\cos kx,\sin kx\text{, }k\in \mathbb{N}\right\} $.
Consider usual spaces $L_{p}$, $1\leq p\leq \infty $, of $p$-integrable
functions $\phi $ on the unit circle $\mathbb{T}$ with the Lebesgue measure $%
dx$, 
\begin{equation*}
\left\Vert \phi \right\Vert _{p}=\left( \int_{\mathbb{T}}\left\vert \phi
\right\vert ^{p}dx\right) ^{\frac{1}{p}}<\infty \text{.}
\end{equation*}%
Let $\phi \in L_{p}$ with the formal Fourier series%
\begin{equation*}
\phi \sim \sum_{k=1}^{\infty }a_{k}\left( \phi \right) \cos kx+b_{k}\left(
\phi \right) \sin kx,
\end{equation*}%
where%
\begin{equation*}
a_{k}\left( \phi \right) =\frac{1}{\pi }\int_{-\pi }^{\pi }\phi \left(
t\right) \cos ktdt,
\end{equation*}%
\begin{equation*}
b_{k}\left( \phi \right) =\frac{1}{\pi }\int_{-\pi }^{\pi }\phi \left(
t\right) \sin ktdt
\end{equation*}%
and%
\begin{equation}
S_{n}\left( \phi ,x\right) =\sum_{k=1}^{n}a_{k}\left( \phi \right) \cos
kx+b_{k}\left( \phi \right) \sin kx  \label{sum}
\end{equation}%
be its $n^{\frac{th}{{}}}$ Fourier sum. We introduce sets of smooth
functions using multipliers $\Lambda =\left\{ \lambda \left( k\right) \text{%
, }k\in \mathbb{N}\right\} $ \cite{C11}. We say that $f\in \Lambda _{\beta
}U_{p}$ if%
\begin{equation}
f\sim \sum_{k=1}^{\infty }\lambda \left( k\right) \left( a_{k}\left( \phi
\right) \cos \left( kx-\frac{\beta \pi }{2}\right) +b_{k}\left( \phi \right)
\sin \left( kx-\frac{\beta \pi }{2}\right) \right) ,  \label{0}
\end{equation}%
where $\phi \in U_{p}=\left\{ \phi \left\vert \left\Vert \phi \right\Vert
_{p}\leq 1\right. \right\} $ is the unit ball in $L_{p}$. If $\beta =0$ then
we write $\Lambda _{\beta }=\Lambda $. If there exists $K\in L_{1}$ such
that 
\begin{equation*}
K\sim \sum_{k=1}^{\infty }\lambda \left( k\right) \cos \left( kx-\frac{\beta
\pi }{2}\right)
\end{equation*}%
then $\Lambda _{\beta }U_{p}$ is the set of functions $f$ \ representable in
the form%
\begin{equation*}
f\left( x\right) =\int_{\mathbb{T}}K\left( x-y\right) \phi \left( y\right)
dy,
\end{equation*}%
i.e. in this case $\Lambda _{\beta }U_{p}=K\ast U_{p}$. Observe that the
smoothness of the classes $\Lambda _{\beta }U_{p}$ is determined by the rate
of decay of the sequence $\Lambda $. In particular, if $\lambda \left(
k\right) =k^{-r}$, $\beta =r$, $r>0$ we get standard Sobolev classes $%
W_{p}^{r}$. If $\lambda \left( k\right) =\exp \left( -\mu k^{\gamma }\right) 
$, $\beta \in \mathbb{R}$, $\mu >0$, $0<\gamma <1$, then the set $\Lambda
_{\beta }U_{p}$ consists of infinitely differentiable functions. In the case 
$\gamma =1$ we get classes of analytic functions. If $\gamma >1$, then we
obtain classes of entire functions. To simplify technical notations we
present our results just in the case $1<p,q<\infty $ and $\beta =0$.

We show here that isometric extensions may decrease the order of linear widths
of $\Lambda U_{p}$ in $L_{p}$ if $1<p<q \leq 2$
in the case of small and finite smoothness, i.e. if%
\begin{equation*}
\lambda \left( k\right) =k^{-r}\text{, }r>\left( \frac{1}{p}-\frac{1}{q}%
\right) _{+},
\end{equation*}%
where $\left( a\right) _{+}=\max \left\{ a,0\right\} $, $a\in \mathbb{R}$.
From the other side, it is shown that in the case of super-small smoothness,
i.e. if 
\begin{equation*}
\lambda \left( k\right) =\phi \left( k\right) k^{-\left( \frac{1}{p}-\frac{1%
}{q}\right) _{+}},
\end{equation*}%
where $\phi \left( k\right) $ is a decreasing function, $\lim_{k\rightarrow
\infty }\phi \left( k\right) =0$ and $\phi \left( k^{s}\right) \asymp \phi
\left( k\right) $ for any fixed $s\in \mathbb{N}$, isometric extensions can
not decrease the order of preabsolute linear widths $\Lambda _{n,m}\left(
\Lambda U_{p},L_{q}\right) $, $0\leq m\leq n$. A typical example of the sequence $\lambda \left(
k\right) $ is given by%
\begin{equation*}
\lambda \left( k\right) =\left( \ln \left( k+1\right) \right) ^{-\varrho
}k^{-\left( \frac{1}{p}-\frac{1}{q}\right) _{+}}\text{, }\varrho >0,k\in 
\mathbb{N}.
\end{equation*}

Similarly, in the case of super-high smoothness, i.e. if 
\begin{equation*}
\lambda \left( k\right) =\exp \left( -\mu n^{\gamma }\right) \text{, }\mu >0%
\text{, }\gamma \geq 1
\end{equation*}%
the order of preabsolute linear widths $\Lambda _{n,m}\left( \Lambda
U_{p},L_{q}\right) $, $\left( p,q\right) \in I$, $0\leq m\leq n$ remains the
same as $\lambda _{n}\left( \Lambda U_{p},L_{q}\right) $ (see (\ref{is})).
In this sense, the results presented here complement the results obtained in 
\cite{C11}.

For easy of notation we will put $a_{n}\gg b_{n}$ for two sequences, if $%
a_{n}>Cb_{n}$ for some $C>0$ and any $n\in \mathbb{N}$ and $a_{n}\asymp
b_{n} $ if $C_{1}b_{n}\leq a_{n}\leq C_{2}b_{n}$ for all $n\in \mathbb{N}$
and some constants $C_{1}>0$ and $C_{2}>0$.

\section{Preabsolute linear widths}

\label{Sec:results}

Our main result significantly reduces the space $\widehat{X}$ \ and gives an
explicit representation of the extension $\mathrm{j}$ which is important for
applications. Consider Banach space $\mathrm{lin}\left( A\right) $ with the
unit ball $A$ and $\left( \mathrm{lin}\left( A\right) \right) ^{\ast }$ its
conjugate with the usual norm $\left\Vert \cdot \right\Vert _{\left( \mathrm{%
lin}\left( A\right) \right) ^{\ast }}$.

\begin{theorem}
\label{THEOREM 1} Let $A\subset X$ be a convex origin symmetric compact, $%
\mathrm{diam}\left( A,X\right) <\infty $, $\phi _{k}\in X^{\ast }$ be such
that%
\begin{equation*}
\sup \left\{ \left\Vert x\right\Vert _{X}\left\vert x\in A\text{, }%
\left\langle x,\phi _{k}\right\rangle =0\text{, }1\leq k\leq n\right.
\right\} \leq d^{n}\left( A,X\right) +\epsilon \text{, }\forall \epsilon >0
\end{equation*}%
and $c_{k}:B\left( X^{\ast }\right) \rightarrow \mathbb{R}$, $1\leq k\leq n$
be the functionals of the best approximation of $\phi \in B\left( X^{\ast
}\right) $ by $\mathrm{lin}\left\{ \phi _{k}\text{, }1\leq k\leq n\right\} $
in $\left\Vert \cdot \right\Vert _{\left( \mathrm{lin}\left( A\right)
\right) ^{\ast }}$. Let%
\begin{equation*}
\left[ c_{k}\right] =\left\{ 
\begin{array}{cc}
0, & c_{k}\in X^{\ast \ast }, \\ 
c_{k}, & c_{k}\notin X^{\ast \ast }.%
\end{array}%
\right. 
\end{equation*}%
Then%
\begin{equation*}
\lambda _{n}\left( \mathrm{j}\left( A\right) ,\overline{X}\right) =\Lambda
_{n}\left( A,X\right) ,
\end{equation*}%
where $\mathrm{j}:X\rightarrow \overline{X}=\mathrm{lin}\left\{ X^{\ast \ast
},\left[ c_{1}\right] ,\cdots ,\left[ c_{n}\right] \right\} \subset \widehat{%
X}$.
\end{theorem}

\begin{proof}
By the definition of linear width, for any extension $\mathrm{j}%
:X\rightarrow \widetilde{X}$ and $\epsilon >0$ there exist $\phi _{k}\in $ $%
\widetilde{X}^{\ast }$and $x_{k}\in \widetilde{X}$, $1\leq k\leq n$ such
that 
\begin{equation*}
\sup_{x\in A}\left\Vert \mathrm{j}\left( x\right)
-\sum_{k=1}^{n}\left\langle \mathrm{j}\left( x\right) ,\phi
_{k}\right\rangle x_{k}\right\Vert _{\widetilde{X}}\leq \lambda _{n}\left( 
\mathrm{j}\left( A\right) ,\widetilde{X}\right) +\epsilon .
\end{equation*}%
Consequently, by the definition of Gelfand widths (see \cite{C1}),%
\begin{equation*}
d^{n}\left( A,X\right) \leq \sup \left\{ \left\Vert \mathrm{j}\left(
x\right) \right\Vert _{\widetilde{X}}\left\vert x\in A\text{, }\left\langle 
\mathrm{j}\left( x\right) ,\phi _{k}\right\rangle =0\text{, }1\leq k\leq
n\right. \right\}
\end{equation*}%
\begin{equation}
\leq \lambda _{n}\left( \mathrm{j}\left( A\right) ,\widetilde{X}\right)
+\epsilon .  \label{0000}
\end{equation}%
Also, by the definition of Gelfand width there are such $\phi _{k}\in
X^{\ast }$, $1\leq k\leq n$ that 
\begin{equation}
\sup \left\{ \left\Vert x\right\Vert _{X}\left\vert x\in A\text{, }%
\left\langle x,\phi _{k}\right\rangle =0\text{, }1\leq k\leq n\right.
\right\} \leq d^{n}\left( A,X\right) +\epsilon  \label{111}
\end{equation}%
for any $\epsilon >0$. Let $\mathrm{lin}\left( A\right) $ be the Banach
space with the unit ball $A$ and $\left( \mathrm{lin}\left( A\right) \right)
^{\ast }$ be its conjugate with the usual norm%
\begin{equation*}
\left\Vert \phi \right\Vert _{\left( \mathrm{lin}\left( A\right) \right)
^{\ast }}=\sup \left\{ \left\vert \left\langle x,\phi \right\rangle
\right\vert \left\vert x\in A\right. \text{ }\right\} .
\end{equation*}%
Then, $\forall \phi \in B\left( X^{\ast }\right) $, by duality and (\ref{111}%
) we get 
\begin{equation*}
\inf \left\{ \left\Vert \phi -\sum_{k=1}^{n}c_{k}\left( \phi \right) \phi
_{k}\right\Vert _{\left( \mathrm{lin}\left( A\right) \right) ^{\ast
}}\left\vert c_{k}\text{, }1\leq k\leq n\right. \right\}
\end{equation*}%
\begin{equation*}
\leq \sup \left\{ \left\vert \left\langle x,\phi \right\rangle \right\vert
\left\vert x\in A\text{, }\left\langle x,\phi _{k}\right\rangle =0\text{, }%
1\leq k\leq n\right. \text{ }\right\}
\end{equation*}%
\begin{equation*}
\leq d^{n}\left( A,X\right) +\epsilon .
\end{equation*}%
Since $\mathrm{diam}\left( A,X\right) <\infty $ then there are such bounded
functions $\phi \mapsto c_{k}\left( \phi \right) $, $1\leq k\leq n$, $\phi
\in B\left( X^{\ast }\right) $ that 
\begin{equation*}
\sup \left\{ \left\Vert \phi -\sum_{k=1}^{n}c_{k}\left( \phi \right) \phi
_{k}\right\Vert _{\left( \mathrm{lin}\left( A\right) \right) ^{\ast
}}\left\vert \phi \in B\left( X^{\ast }\right) \right. \right\} \leq
d^{n}\left( A,X\right) +\epsilon ,
\end{equation*}%
or%
\begin{equation*}
\sup_{\phi \in B\left( X^{\ast }\right) }\sup_{x\in A}\left\vert
\left\langle x,\phi \right\rangle -\sum_{k=1}^{n}c_{k}\left( \phi \right)
\left\langle x,\phi _{k}\right\rangle \right\vert
\end{equation*}%
\begin{equation*}
=\sup_{x\in A}\sup_{\phi \in B\left( X^{\ast }\right) }\left\vert
\left\langle x,\phi \right\rangle -\sum_{k=1}^{n}c_{k}\left( \phi \right)
\left\langle x,\phi _{k}\right\rangle \right\vert
\end{equation*}%
\begin{equation}
=\sup_{x\in A}\left\Vert \mathrm{j}\left( x\right)
-\sum_{k=1}^{n}\left\langle x,\phi _{k}\right\rangle c_{k}\right\Vert _{%
\widehat{X}}\leq d^{n}\left( A,X\right) +\epsilon .  \label{2222}
\end{equation}%
Since $\phi \in B\left( X^{\ast }\right) $ then $\left\langle x,\phi
\right\rangle \in X^{\ast \ast }$ and $c_{k}\in \widehat{X}$. Consequently, 
\begin{equation*}
\left\langle x,\phi \right\rangle -\sum_{k=1}^{n}\left\langle x,\phi
_{k}\right\rangle c_{k}\in \mathrm{lin}\left\{ X^{\ast \ast }\text{, }c_{k}%
\text{, }1\leq k\leq n\right\}
\end{equation*}%
\begin{equation*}
=\mathrm{lin}\left\{ X^{\ast \ast }\text{, }\left[ c_{k}\right] \text{, }%
1\leq k\leq n\right\} ,
\end{equation*}%
where $c_{k}$, $1\leq k\leq n$ are defined by (\ref{2222}). Observe that if
among $c_{k}$, $1\leq k\leq n$ there are linear functionals $c_{s}$ on $%
X^{\ast }$ then $c_{s}\in X^{\ast \ast }$, or $\left[ c_{s}\right] =0$.
Comparing (\ref{0000}) and (\ref{2222}) we get the proof.
\end{proof}

Theorem \ref{THEOREM 1} allows us to introduce the following notion.

\begin{definition}
\label{DEF 1} Let $\left( X,\left\Vert \cdot \right\Vert _{X}\right) $ be a
Banach space and $A\subset X$ be a compact, convex and origin symmetric set
in $X$. The $m$-preabsolute linear $n$-width of $A$ in $X$ is defined by%
\begin{equation*}
\Lambda _{n,m}\left( A,X\right) =\inf \lambda _{n}\left( \mathrm{j}%
_{m}\left( A\right) ,\widetilde{X}\right) \text{, }0\leq m\leq n,
\end{equation*}%
where $\inf $ is taken over all isometric extensions 
\begin{equation*}
\mathrm{j}_{m}:X\rightarrow \widetilde{X}=\mathrm{lin}\left\{ X,c_{k}\text{, 
}k\leq m\text{, }c_{0}=0\right\}
\end{equation*}%
and%
\begin{equation*}
c_{k}:B\left( X^{\ast }\right) \rightarrow \mathbb{R}\text{, }k\leq m\text{. 
}
\end{equation*}
\end{definition}

Observe that, by Theorem \ref{THEOREM 1}, 
\begin{equation*}
\lambda _{n}\left( A,X\right) =\Lambda _{n,0}\left( A,X\right) \geq \Lambda
_{n,1}\left( A,X\right) \geq \cdots \geq \Lambda _{n,n}\left( A,X\right)
\end{equation*}%
\begin{equation}
=\Lambda _{n}\left( A,X\right) =d^{n}\left( A,X\right) .  \label{ooo}
\end{equation}

\section{Examples and application}

\label{Sec:examples}

In this section we consider several motivating examples in the case $%
1<p<q\leq 2$ to underline the dependence of preabsolute widths on
smoothness. In particular, it is shown that in the cases of super-small and
super-high smoothness isometric extensions can not decrease the order of
linear widths.

\begin{theorem}
\bigskip \label{THEOREM 2} 1. Let $1<p<q\leq 2$ and 
\begin{equation*}
\lambda \left( k\right) =\varphi \left( k\right) k^{-\left( \frac{1}{p}-%
\frac{1}{q}\right) _{+}},k\in \mathbb{N}
\end{equation*}%
in (\ref{0}), where $\left( a\right) _{+}=\max \left\{ a,0\right\} $, $%
\varphi \left( k\right) $ is a decreasing function, $\lim_{k\rightarrow
\infty }\varphi \left( k\right) =0$ and $\varphi \left( k^{s}\right) \asymp
\varphi \left( k\right) $ for any fixed $s>0$ (i.e. the case of super-small
smoothness). Then%
\begin{equation}
\Lambda _{n,m}\left( \Lambda U_{p},L_{q}\right) \asymp \varphi \left(
n\right) \text{, }1<p,q<\infty \text{, }0\leq m\leq n.  \label{sss}
\end{equation}%
2. If $\lambda \left( k\right) =k^{-r}$ where%
\begin{equation*}
\frac{1}{p}-\frac{1}{q}<r<\frac{1}{2}\left( \frac{1}{p}-\frac{1}{q}\right)
/\left( \frac{1}{p}-\frac{1}{2}\right) \text{, }1<p<q\leq 2
\end{equation*}%
(i.e. the case of small smoothness) then%
\begin{equation}
n^{\frac{p}{2\left( p-1\right) }\left( -r+\frac{1}{p}-\frac{1}{q}\right)
}\ll \Lambda _{n,m}\left( W_{p}^{r},L_{q}\right) \ll n^{-r+\frac{1}{p}-\frac{%
1}{q}}\text{, }0\leq m\leq n.  \label{ss}
\end{equation}

3. If $\lambda \left( k\right) =k^{-r}$ where%
\begin{equation*}
r>\frac{1}{2}\left( \frac{1}{p}-\frac{1}{q}\right) /\left( \frac{1}{p}-\frac{%
1}{2}\right)
\end{equation*}
(i.e. the case of finite smoothness) then%
\begin{equation}
n^{-r}\ll \Lambda _{n,m}\left( W_{p}^{r},L_{q}\right) \ll n^{-r+\frac{1}{p}-%
\frac{1}{q}}\text{, }0\leq m\leq n\text{.}  \label{fs}
\end{equation}%
4. Let $\lambda \left( k\right) =\exp \left( \mu k^{\gamma }\right) $, $\mu
>0$, $0<\gamma <1$ in (\ref{0}) (i.e. the case of infinite smoothness) then%
\begin{equation}
\exp \left( -\mu n^{\gamma }\right) \ll \Lambda _{n,m}\left( \Lambda
U_{p},L_{q}\right) \ll \exp \left( -\mu n^{\gamma }\right) n^{\left(
1-\gamma \right) \left( \frac{1}{p}-\frac{1}{q}\right) }\text{, }0\leq m\leq
n.  \label{is}
\end{equation}%
5. Let $\lambda \left( k\right) =\exp \left( \mu k^{\gamma }\right) $, $\mu
>0$, $\gamma \geq 1$ in (\ref{0}) (i.e. the case of super-high smoothness)
then 
\begin{equation}
\Lambda _{n,m}\left( \Lambda U_{p},L_{q}\right) \asymp \exp \left( -\mu
n^{\gamma }\right) \text{, }\mu >0\text{, }\gamma \geq 1\text{, }0\leq m\leq
n.  \label{shs}
\end{equation}
\end{theorem}

\begin{proof}
Let us consider the case of super-small smoothness (\ref{sss}).
It was
shown in \cite{C11} that in this case
\begin{equation*}
d^{n}\left( \Lambda U_{p},L_{q}\right) \asymp \lambda _{n}\left( \Lambda
U_{p},L_{q}\right) 
\end{equation*}

\begin{equation*}
\asymp \sup_{f\in \Lambda U_{p}}\left\Vert f-S_{n}\left( f\right)
\right\Vert _{q}\asymp \varphi \left( n\right) \text{.}
\end{equation*}%
Consequently, in this case 
\begin{equation}
\Lambda _{n,m}\left( \Lambda U_{p},L_{q}\right) \asymp \varphi \left(
n\right) \text{, }1<p,q<\infty 
\end{equation}%
for any $0\leq m\leq n$.

Let $\Lambda U_{p}=W_{p}^{r}$ be Sobolev class. In this case $\lambda \left(
k\right) =k^{-r}$. We show (\ref{fs}). It is known \cite{C11} that%
\begin{equation}
\lambda _{n}\left( W_{p}^{r},L_{q}\right) \ll \sup_{f\in
W_{p}^{r}}\left\Vert f-S_{n}\left( f\right) \right\Vert _{q}\ll n^{-r+\frac{1%
}{p}-\frac{1}{q}}\text{, }1<p<q\leq 2\text{, }r>\frac{1}{p}-\frac{1}{q},
\label{upper}
\end{equation}%
where $S_{n}$ is defined by (\ref{sum}).
By the Theorem \ref{THEOREM 2}, (\ref{ooo}) and (\ref{dl1}) for any $0\leq
m\leq n$ we have%
\begin{equation*}
\Lambda _{n,m}\left( W_{p}^{r},L_{q}\right) \geq \Lambda _{n,n}\left(
W_{p}^{r},L_{q}\right) =\Lambda _{n}\left( W_{p}^{r},L_{q}\right) 
\end{equation*}%
\begin{equation*}
=d^{n}\left( W_{p}^{r},L_{q}\right) =d_{n}\left( W_{q^{^{\prime
}}}^{r},L_{p^{^{\prime }}}\right) ,
\end{equation*}%
where 
\begin{equation*}
p^{^{\prime }}=\left\{ 
\begin{array}{cc}
\frac{p}{p-1}, & 1<p<\infty , \\ 
1, & p=\infty , \\ 
\infty , & p=1.%
\end{array}%
\right. 
\end{equation*}%
Since $1<p<q\leq 2$ then $2\leq q^{^{\prime }}<p^{^{\prime }}<\infty $ and if%
\begin{equation}
r>\frac{1}{2}\left( \frac{1}{q^{^{\prime }}}-\frac{1}{p^{^{\prime }}}\right)
/\left( \frac{1}{2}-\frac{1}{p^{^{\prime }}}\right) =\frac{1}{2}\left( \frac{%
1}{p}-\frac{1}{q}\right) /\left( \frac{1}{p}-\frac{1}{2}\right)   \label{pq1}
\end{equation}%
then%
\begin{equation*}
d_{n}\left( W_{q^{^{\prime }}}^{r},L_{p^{^{\prime }}}\right) \asymp n^{-r}.
\end{equation*}%
Clearly, 
\begin{equation*}
\frac{1}{2}\left( \frac{1}{p}-\frac{1}{q}\right) /\left( \frac{1}{p}-\frac{1%
}{2}\right) >\frac{1}{p}-\frac{1}{q}.
\end{equation*}%
Hence 
\begin{equation*}
n^{-r}\ll \Lambda _{n,m}\left( W_{p}^{r},L_{q}\right) \ll n^{-r+\frac{1}{p}-%
\frac{1}{q}}\text{, }0\leq m\leq n
\end{equation*}%
if (\ref{pq1}) is satisfied. This proves (\ref{fs}).

To show (\ref{ss}) we remark that if $2\leq p<q<\infty $ then \cite{GALEEV}, 
\cite{C23}, \cite{C11} 
\begin{equation*}
d_{n}\left( W_{p}^{r},L_{q}\right) \asymp n^{\frac{q}{2}\left( -r+\frac{1}{p}%
-\frac{1}{q}\right) },
\end{equation*}%
where%
\begin{equation*}
\frac{1}{p}-\frac{1}{q}<r<\frac{1}{2}\left( \frac{1}{p}-\frac{1}{q}\right)
/\left( \frac{1}{2}-\frac{1}{q}\right) .
\end{equation*}%
Consequently, by (\ref{ooo}) and (\ref{dl1}) we get 
\begin{equation*}
\Lambda _{n,m}\left( W_{p}^{r},L_{q}\right) \geq d^{n}\left(
W_{p}^{r},L_{q}\right) =d_{n}\left( W_{q^{^{\prime }}}^{r},L_{p^{^{\prime
}}}\right) 
\end{equation*}%
\begin{equation*}
\asymp n^{\frac{p^{^{\prime }}}{2}\left( -r+\frac{1}{q^{^{\prime }}}-\frac{1%
}{p^{^{\prime }}}\right) }=n^{\frac{p}{2\left( p-1\right) }\left( -r+\frac{1%
}{p}-\frac{1}{q}\right) }.
\end{equation*}%
The respective upper bounds follow from (\ref{upper}). This proves (\ref{ss}%
).

The case (\ref{is}) can be treated similarly. Namely, since 
\begin{equation*}
\lambda _{2n}\left( \Lambda U_{p},L_{q}\right) \ll \sup_{f\in
W_{p}^{r}}\left\Vert f-S_{n}\left( f\right) \right\Vert _{q}\ll \exp \left(
-\mu n^{\gamma }\right) n^{\left( 1-\gamma \right) \left( \frac{1}{p}-\frac{1%
}{q}\right) _{+}}\text{, }
\end{equation*}%
where%
\begin{equation*}
1<p<q\leq 2\text{, }\mu >0\text{, }0<\gamma <1.
\end{equation*}%
\cite{C11}, \cite{KKK}, \cite{KKK1}, \cite{KKK2}, \cite{KKK3}
and by (\ref{dl1})  
\begin{equation*}
d^{2n}\left( \Lambda U_{p},L_{q}\right) \asymp d_{2n}\left( \Lambda
U_{p^{^{\prime }}},L_{q^{^{\prime }}}\right) \asymp \exp \left( -\mu
n^{\gamma }\right) ,
\end{equation*}%
\begin{equation*}
2\leq p^{^{\prime }}<q^{^{\prime }}<\infty 
\end{equation*}%
then%
\begin{equation*}
\exp \left( -\mu n^{\gamma }\right) \ll \Lambda _{2n,2m}\left( \Lambda
U_{p},L_{q}\right) \ll \exp \left( -\mu n^{\gamma }\right) n^{\left(
1-\gamma \right) \left( \frac{1}{p}-\frac{1}{q}\right) }\text{, }0\leq m\leq
n.
\end{equation*}%
Finally, consider the case of super-high smoothness (\ref{shs}). Namely, if%
\begin{equation*}
\lambda \left( k\right) =\exp \left( \mu k^{\gamma }\right) \text{, }\mu >0%
\text{, }\gamma \geq 1.
\end{equation*}%
In this case \cite{KKK}%
\begin{equation*}
\lambda _{2n}\left( \Lambda U_{p},L_{q}\right) \ll \sup_{f\in
W_{p}^{r}}\left\Vert f-S_{n}\left( f\right) \right\Vert _{q}\ll \exp \left(
-\mu n^{\gamma }\right) \text{, }1<p,q<\infty 
\end{equation*}%
and%
\begin{equation*}
d^{2n}\left( \Lambda U_{p},L_{q}\right) \asymp d_{2n}\left( \Lambda
U_{p^{^{\prime }}},L_{q^{^{\prime }}}\right) \asymp \exp \left( -\mu
n^{\gamma }\right) \text{, }1<p,q<\infty .
\end{equation*}%
Consequently,

\begin{equation*}
\Lambda _{n,m}\left( \Lambda U_{p},L_{q}\right) \asymp \exp \left( -\mu
n^{\gamma }\right) \text{, }\mu >0\text{, }\gamma \geq 1
\end{equation*}

for any $0\leq m\leq n$.
\end{proof}

\section*{Acknowledgements}

The Author wishes to thank the organisers of the $2^{\frac{nd}{{}}}$
Gaussian Symposium, Munich, 1993, where Theorem 1 was reported by the
author. Also, I wish to thank referees and Communicating Editor for valuable
comments.

\end{document}